\documentclass{amsart}
\usepackage{amsmath, amssymb}
\usepackage{tikz-cd}
\usepackage{array}
\usepackage{float}

\usepackage{amsmath, amssymb, amsthm, enumitem}
\usepackage{mathrsfs}

\newtheorem{theorem}{Theorem}[section]
\newtheorem{corollary}[theorem]{Corollary}
\newtheorem{lemma}[theorem]{Lemma}
\newtheorem{proposition}[theorem]{Proposition}

\theoremstyle{definition}
\newtheorem{definition}[theorem]{Definition}
\newtheorem{example}[theorem]{Example}

\theoremstyle{remark}
\newtheorem{remark}[theorem]{Remark}

\usepackage{thmtools}
\usepackage{thm-restate}
\usepackage{hyperref}
\usepackage{cleveref}

\newcommand{\PP}{\mathbb{P}}



\begin{document}

\pagestyle{plain}
\title{Algebraic hyperbolicity of very general hypersurfaces in weighted projective spaces}

\author[J. Wang]{Jiahe Wang}
\address{UCLA Mathematics Department, Box 951555, Los Angeles, CA 90095-1555, USA
}
\email{jiahewang@math.ucla.edu}

\begin{abstract}
We provide a bound for $m$ such that the zero locus of a very general section of an $m$-multiple of
some ample line bundle on a weighted projective space with isolated singularities is algebraically hyperbolic.
\end{abstract}

\maketitle

\section{Introduction}
\begin{definition}
    Given a complex projective variety $X$, we say that $X$ is \emph{algebraically hyperbolic} if there exists an $\epsilon > 0$ and an ample line bundle $P$ such that for any integral curve $C \subset X$, we have that
    \[
    2g(C) - 2 \geq \epsilon \cdot \deg_P C
    \]
    where $g(C)$ is the geometric genus of $C$.
\end{definition}

 Algebraic hyperbolicity was introduced as an algebraic analogue to Kobayashi hyperbolicity for complex manifolds \cite{Demailly1997}. A complex manifold \(X\) is said to be Kobayashi hyperbolic if its Kobayashi pseudometric is nondegenerate, and Brody hyperbolic if every entire map \(f:\mathbb{C}\to X\) is constant. For smooth projective varieties, Kobayashi hyperbolicity implies algebraic hyperbolicity and the converse is conjectured to be true \cite{Demailly1997}. Brody hyperbolicity is in general weaker than Kobayashi hyperbolicity but is equivalent to Kobayashi hyperbolicity for compact manifolds \cite{Brody1978}.

The algebraic hyperbolicity of very general hypersurfaces of a smooth complex projective  variety $A$ with a group action is well-studied. In this paper, we adopt the normal bundle technique developed by Ein \cite{Ein1988, Ein1991}, Pacienza \cite{Pacienza2003}, Voisin \cite{Voisin1996, Voisin1998}, Coskun and Riedl \cite{CoskunRiedl2019,coskun2019algebraichyperbolicitygeneralsurfaces}, and Yeong \cite{yeong2022algebraichyperbolicitygeneralhypersurfaces}; see also \cite{mioranci2025algebraichyperbolicitygeneralhypersurfaces,moraga2025hyperbolicityconjectureadjointbundles} for recent works. The idea is the following: Let $\mathcal E$ be a globally generated vector bundle on a smooth projective variety $A$ admitting a Zariski-open subset on which an algebraic group acts transitively, and let $X$ be the zero locus of a very general section of $\mathcal E$. Let $C \subset X$ be a curve in $X$. Then the normal bundle $N_{C/X}$ is related to the genus by $2g-2-K_X \cdot C = \deg N_{C/X}$. If we can find a lower bound for the degree of $N_{C/X}$ by the intersection number of an ample line bundle $P$ with the curve and on the other hand express $K_X$ in terms of the ample line bundle $P$, then we get an expression of the form $2g-2 \geq a (P \cdot C) = a \deg_P C$, for some constant $a$ possibly independent of $C$. By examining the positivity of $a$, we get a sufficient condition for algebraic hyperbolicity of $X$. To bound the degree of the normal bundle, we will see that under the construction, there exists a surjection from the syzygy bundle of $\mathcal E$ to some subsheaf of $N_{C/X}$, which gives a lower bound for $\deg N_{C/X}$. By considering a section-dominating collection of $\mathcal E$, we can further improve this bound.

In \cite{coskun2019algebraichyperbolicitygeneralsurfaces}, Coskun and Riedl develop and apply this technique to very general surfaces in threefolds. In particular, the authors apply the technique to the resolution of $\mathbb{P}(1,1,1,n)$ for $n\geq1$, $f:\widetilde{\mathbb{P}} \to \mathbb{P}(1,1,1,n)$. Let $H = f^*\mathcal{O}(n)$, the authors give a bound for $m$ such that the zero locus of a very general section of $mH$ is algebraically hyperbolic.

\begin{proposition}[{\cite[Lemma 3.9, Proposition 3.10]{coskun2019algebraichyperbolicitygeneralsurfaces}}]
A very general surface $X \in |mH|$ is algebraically hyperbolic if $m \geq 4$ and $n \geq 2$; or $m = 3$ and $n \geq 4$; or $m = 2$ and $n \geq 5$. And $X$ will not be algebraically hyperbolic if $n = 1$, $m\leq4$ or $m = 2$, $n \leq 4$.
\end{proposition}

In this paper, we see that the same setup and arguments apply to an arbitrary weighted projective space $\mathbb P(a_0,...,a_n)$ of dimension $n\geq 3$ with isolated singularities, and we yield the following results. We assume that $\mathbb P(a_0,...,a_n)$ is well-formed (see section 3).
\begin{proposition}
If $m > \frac{a_0 + \dots + a_n}{a_0a_1... a_n} + (n-2)$, then a very general hypersurface $X$ of $|mH|$ is algebraically hyperbolic outside the toric boundary.
\end{proposition}
\begin{proposition}
For a weighted projective $3$-fold $\mathbb{P}(a_0,a_1,a_2,a_3)$ with isolated singularities, if $\mathbb{P}(a_0,a_1,a_2,a_3) \neq \mathbb{P}(1,1,1,n)$ or $\mathbb{P}(1,1,2,3)$, then a very general surface $X\in |mH|$ is algebraically hyperbolic if $m \geq 2$. For $\mathbb{P}(1,1,2,3)$, a very general surface $X\in|mH|$ is algebraically hyperbolic if $m\geq 3$.
\end{proposition}
\begin{proposition}
Let $\mathbb{P} = \mathbb{P}(a_0,\dots,a_n)$ be a weighted projective space with isolated singularities. Let
\[
\Theta := \max_{\substack{I \subset \{0,\dots,n\} \\ |I|\ge 4}} \left\{\frac1{\prod_{i\notin I}a_i} \left( \frac{\sum_{i \in I} a_i}{\prod_{i\in I} a_i} + (|I|-3) \right)\right\}.
\]
Then for every $m > \Theta$, a very general hypersurface $X \in |mH|$ is algebraically hyperbolic.
\end{proposition}

\subsection*{Acknowledgments}
I sincerely thank Wern Yeong for bringing up this project and helpful conversations. I am grateful to Burt Totaro for helpful conversations and comments on my draft. I am grateful to Eric Riedl and Sixuan Lou for helpful comments.

\section{The normal bundle technique}

We briefly recall the normal bundle technique developed and formulated in Ein \cite{Ein1988, Ein1991}, Pacienza \cite{Pacienza2003}, Voisin \cite{Voisin1996, Voisin1998}, Coskun and Riedl \cite{CoskunRiedl2019,coskun2019algebraichyperbolicitygeneralsurfaces}, and Yeong \cite{yeong2022algebraichyperbolicitygeneralhypersurfaces}. We closely follow the formulation in \cite{coskun2019algebraichyperbolicitygeneralsurfaces}. The technique bounds the degree of the normal bundle by the degree of syzygy bundle, which can be further bounded using the syzygy bundles of the line bundles in a section-dominating collection.

Let $A$ be a smooth complex projective variety of dimension $n$ that admits a Zariski-open set $A_0$ with a transitive group action by some algebraic group $G$. For example, $A$ can be a homogeneous variety such as $\mathbb{P}^n$, Grassmannians, or flag varieties, or $A$ can be a smooth toric variety. Let $\mathcal E$ be a globally generated vector bundle invariant under $G$ on $A$ of rank $r < n-1$. Let $X$ be the zero locus of a very general section of $\mathcal E$. If $X$ contains a curve of degree $e$ and genus $g$ that intersects with $A_0$,
let $V=H^0(A,\mathcal E)$,
$\mathcal X_1 \to V$ the universal hypersurface over $V$,
$\mathcal H \to V$ the relative Hilbert scheme with universal curve $\mathcal Y_1 \to \mathcal X_1$ where the general fiber $\mathcal Y_1\to \mathcal H$ is a geometric genus $g$ curve of degree $e$. 
We can find a $G$-invariant subvariety $U$ in $\mathcal H$ such that $U \to V$ is étale.
Let  $\mathcal Y_2\to U$ be the restriction of $\mathcal Y_1$ to $U$. By taking a resolution of the general fiber (and possibly further restricting $U$), we get a smooth family $\mathcal Y \to U$ whose fibers are smooth curves of genus $g$. Let $\mathcal X$ be the pullback of the family of $\mathcal X_1$ over $U$, then we get 
$\pi_1: \mathcal X \to U$, $\pi_2: \mathcal X \to A$, and $h: \mathcal Y \to \mathcal X$ which is generically injective.
\begin{definition}
Let $\mathcal{E}$ be a globally generated vector bundle on $A$. 
The \emph{syzygy bundle} or \emph{Lazarsfeld--Mukai bundle} associated to $\mathcal{E}$ is the vector bundle $M_{\mathcal{E}}$ defined by the short exact sequence
\[
0 \longrightarrow M_{\mathcal{E}} \longrightarrow H^0(A, \mathcal{E}) \otimes \mathcal{O}_A 
\stackrel{\mathrm{ev}}{\longrightarrow} \mathcal{E} \longrightarrow 0,
\]
where $\mathrm{ev}$ is the evaluation map.
\end{definition}
Let $t$ be a general element of $U$. Let $Y_t$ be the fiber of $\mathcal Y$ over $t$ and $X_t$ be the fiber of $\mathcal X$
over $t$. Let $h_t
: Y_t \to X_t$ be the restriction of $h$ to $Y_t$
.We have the following exact sequences and properties.

\vspace{1em}

\begin{tikzcd}
  & 0 \arrow[d] & 0 \arrow[d] & 0  \arrow[d]& \\
0 \arrow[r] & S \arrow[r] \arrow[d] & h^*\pi_2^* T_A \arrow[r] \arrow[d] & T \arrow[r] \arrow[d] & 0 \\
0 \arrow[r] & T_{\mathcal Y} \arrow[r] \arrow[d] & h^*T_{\mathcal X} \arrow[r] \arrow[d] & N_{h/\mathcal X} \arrow[r] \arrow[d] & 0 \\
0 \arrow[r] & T_{\mathcal Y/A} \arrow[r] \arrow[d] & h^*T_{\mathcal X/A} \arrow[r] \arrow[d] & \mathcal{K} \arrow[r] \arrow[d] & 0 \\
  & 0 & 0 & 0 &
\end{tikzcd}

\begin{proposition}[{\cite[Proposition~2.1]{coskun2019algebraichyperbolicitygeneralsurfaces}}]\leavevmode 
\begin{enumerate}
    \item $N_{h_t/X_t} \cong  N_{h/\mathcal X}|_{Y_t}$.
    \item $T_{\mathcal X/A} \cong  \pi_2^* M_{\mathcal E}$.
    \item If $A_0 = A$, then $N_{h/\mathcal X}$ is the cokernel of the map of vertical tangent spaces 
    $T_{\mathcal Y/A} \to T_{\mathcal X/A}$. If $A_0 \ne A$, then the cokernel of the map 
    $T_{\mathcal Y/A} \to T_{\mathcal X/A}$ is a sheaf $\mathcal{K}$ that injects into $N_{h/\mathcal X}$ 
    with torsion cokernel.
\end{enumerate}
\end{proposition}

 \begin{remark}
 \begin{itemize}
     \item The transitive G-action on $A_0$ is essential: $\pi_2 \circ h$ dominates $A_0$ as $Y$ is stable under G-action, so $T_{Y}\to h^{*}\pi_{2}^{*}T_{A}$ is generically surjective over $A_{0}$. Consequently, the cokernel $\mathcal K$ of $T_{Y/A}\to h^{*}T_{X/A}$ injects into $N_{h/X}$ with torsion cokernel.

\item By taking an étale cover, we only take a subvariety of curves of fixed genus and degree into consideration but this suffices as eventually the strategy will yield some $\epsilon$ independent of $C$. If there is any $\epsilon>0$ that satisfy the inequality $2g-2 \geq \epsilon\cdot \deg C$ by curves in the subvariety, then this inequality will also be satisfied by the curves with the same genus and degree. Also we note that this $\epsilon$ is independent of genus and degree.
\end{itemize} 
\end{remark}

\begin{definition} [{\cite[Definition~2.3]{coskun2019algebraichyperbolicitygeneralsurfaces}}]
Let $\mathcal{E}$ be a vector bundle on $A$.
A collection of non-trivial, globally generated line bundles
$L_1, \dots, L_u$ is called a \emph{section-dominating collection}
of line bundles for $\mathcal{E}$ if
\begin{enumerate}
    \item $\mathcal{E} \otimes L_i^{\vee}$ is globally generated for every $1 \leq i \leq u$, and
    \item the map
    \[
    \bigoplus_{i=1}^{u}
    \bigl(H^{0}(L_i \otimes \mathcal{I}_{p}) \otimes H^{0}(\mathcal{E} \otimes L_i^{\vee})\bigr)
    \longrightarrow H^{0}(\mathcal{E} \otimes \mathcal{I}_{p})
    \]
    is surjective at every point $p \in A$.
\end{enumerate}
\end{definition}

\begin{example}[{\cite[Example 2.4]{coskun2019algebraichyperbolicitygeneralsurfaces}}]
Let $A = \mathbb{P}^2 \times \mathbb{P}^1$ with $\mathcal O(H_i), i = 1,2$ the pullbacks of the $\mathcal O(1)$ under projections. 
Let $\mathcal{E} =\mathcal O(aH_1 + bH_2)$ with $a,b > 0$. Then $\mathcal O(H_1)$ and $\mathcal O (H_2)$ is a section-dominating collection for $\mathcal{E}$. This is because $ \mathbb{P}^2 \times \mathbb{P}^1$ is a homogeneous space, so we can take $p$ to be $([0,0,1],[0,1])$. 
Then $H^0\!\bigl((aH_1 + bH_2) \otimes \mathcal{I}_p\bigr)$
is the set of polynomials of bidegree $(a,b)$ in variables $x,y,z$ and $s,t$ with each monomial being divisible by $x,y$ or $s$. Because the only non-vanishing monomial at the point is $z^at^b$, it has to have zero coefficient as it cannot be canceled out by other terms. 
So
$
H^0(H_1 \otimes \mathcal{I}_p) \otimes H^0((a-1)H_1 + bH_2) \oplus
H^0(H_2 \otimes \mathcal{I}_p) \otimes H^0(aH_1 + (b-1)H_2)
\to
H^0\!\bigl((aH_1 + bH_2)\otimes \mathcal{I}_p\bigr)
$ is surjective.
\end{example}

If we have a section-dominating collection of line bundles for $\mathcal{E}$, then we get a surjection to $M_{\mathcal{E}}$:

\medskip

\begin{proposition}[{\cite[Proposition~2.6]{coskun2019algebraichyperbolicitygeneralsurfaces}}]
Let $\mathcal{E}$ be a globally generated vector bundle and $M_{\mathcal{E}}$
the Lazarsfeld--Mukai bundle associated to $\mathcal{E}$.
Let $L_1, \dots, L_u$ be a section-dominating collection of line bundles for $\mathcal{E}$.
Then for some integers $s$, there is a surjection
$$
\bigoplus_{i=1}^{u} M_{L_i}^{\oplus s} \longrightarrow M_{\mathcal{E}}
$$
induced by multiplication by some choice of basis elements of
$H^{0}(\mathcal{E} \otimes L_i^{\vee})$ for $0 \leq i \leq u$.
\medskip

\end{proposition}

We examine a specific scenario which is relevant to weighted projective spaces. We consider $\mathcal E = mL, m\geq1$ for some globally generated line bundle $L$. We examine when $mL$ is section dominated by $L$.
\begin{definition}
    We say a line bundle $L$ on $A$ is \emph{normally generated} if its section ring $R(A, L):= \bigoplus_{m \geq 0}  H^0(A,mL)$ is generated in degree 1, equivalently the map $\operatorname{Sym}^m H^0(A,L)\to H^0(A,mL)$ is surjective for all $m\geq 2$.
 \end{definition} 

\begin{lemma}\label{section}
Given a globally generated and normally generated line bundle $L$ on $A$, we have that $L$ is a section-dominating collection of $mL$ for $m\geq 1$.
\end{lemma}
\begin{proof}
Since $L$ is normally generated, we have the surjection $H^0(A,L) \otimes H^0(A,(m-1)L) 
\to H^0(A,mL)$. For any point $p\in A$, we have to show that $H^0(A,L\otimes \mathcal{I}_p) \otimes H^0(A,(m-1)L) 
\to H^0(A,mL\otimes \mathcal{I}_p)$ is surjective. But this is true, as we consider the evaluation map $H^0(A, L)\to \mathbb{C}$, then the kernel of this map is a hyperplane of $H^0(A, L)$ since $L$ is globally generated. We form a basis of $H^0(A,L)$, $g_1,\dots,g_{k-1},s$ where $g_1,..g_{k-1}$ is a basis of the hyperplane and $s(p)\neq 0$. Then for any $f \in H^0(A,mL\otimes \mathcal{I}_p)$, $f$ can be written as sum of products of sections in $H^0(A,L)$. Write each section in $H^0(A,L)$ in terms of the basis and expand the terms, then we have all monomial terms contain some $g_i$, except for the single term $s^m$. So the coefficient of $s^m$ has to be zero, so each monomial term contains some $g_i$ and the map is surjective.
\end{proof}

\begin{lemma}[{\cite[Proposition~2.7]{coskun2019algebraichyperbolicitygeneralsurfaces}}]\label{lem2}
Given a surjection from $M_L$ to a vector bundle $N$ on a curve $C$, we have that $\deg N \geq -\operatorname{rank}(N)\deg L|_C$.
\end{lemma}

\begin{proof}
The authors in \cite{coskun2019algebraichyperbolicitygeneralsurfaces} considered the case when $N$ has rank 1, but the same proof applies to higher ranks. We recall the proof. Consider the short exact sequence
\[
0 \to M_L \to \mathcal{O} \otimes H^0(L) \to L \to 0.
\]
Taking the second wedge power of the sequence, we get
\[
0 \to \wedge^2 M_L \to \wedge^2(\mathcal{O} \otimes H^0(L)) \to M_L(L) \to 0.
\]
Since $\wedge^2(\mathcal{O} \otimes H^0(L))$ is trivial, we have $M_L(L)$ is globally generated, and hence, so is $N(L)$. Since the degree of $N(L)$ must be non-negative, we have $\deg N \geq -\operatorname{rank}(N)\deg L|_C$.
\end{proof}

\begin{lemma}\label{lemma}
If $mL$ is section-dominated by $L$ and we have a surjection from $M_{mL}$ to a vector bundle $N$, then $\deg N \geq -\operatorname{rank}(N)\deg L|_C$.
\end{lemma}

\begin{proof}
We have the surjections $(M_L)^{\oplus s} \to M_{mL}\to N$.

We have the exact sequence 
\[
0 \longrightarrow (\wedge^2 M_L)^{\oplus s} 
\longrightarrow (\wedge^2 (H^0(L) \otimes \mathcal{O}_A))^{\oplus s} 
\longrightarrow (M_L)^{\oplus s} \otimes L 
\longrightarrow 0.
\] 

Since $(\wedge^2 (H^0(L) \otimes \mathcal{O}_X))^{\oplus s}$ is trivial so it is globally generated, so 
$(M_L)^{\oplus s} \otimes L$ and $N \otimes L$ are globally generated. So $\deg(N\otimes L) = \operatorname{rank}(N)\deg(L)+\deg (N) \geq 0$, so $\deg(N)\geq- \operatorname{rank}(N) \deg L|_C$.
\end{proof}
\begin{remark} If we only consider the surjection $M_{mL}\to N$ without section-dominating collection, then we get that $\deg N \geq -\operatorname{rank}(N)m \deg L|_C$. The 
$m$-dependence makes it ineffective on weighted projective spaces as in the next section.
\end{remark}

Recall the construction of the universal family. Let $t$ be a general element of $U$. Let $Y_t$ be the fiber of $\mathcal Y$ over $t$ and $X_t$ be the fiber of $\mathcal X$. Similar to Theorem 1.2 in \cite{coskun2019algebraichyperbolicitygeneralsurfaces}, we get:
\begin{lemma}\label{lem}
Let $\mathcal E = mL$ on $A$ with $\operatorname{dim}A = n$, such that $L$ is a section-dominating collection of $mL$. 
We have that \[2g(Y_t)- 2 - K_{X_t} \cdot h_t(Y_t) =\deg(N_{h_t/X_t}) \geq \deg(\mathcal{K}|_{Y_t})\geq -(n-2)\deg L|_{Y_t}
.\]

\begin{proof}
The first equality and inequality are shown in \cite[Lemma 2.2]{coskun2019algebraichyperbolicitygeneralsurfaces}: The first equality is from the exact sequence
$
0 \to T_{Y_t} \to h_t^* T_{X_t} \to N_{h_t/X_t} \to 0.
$ For the inequality, we have a surjection from $h^*\pi_2^*M_{\mathcal E}$ onto $\mathcal{K}$, which injects into $N_{h/X}$. Restricting to $Y_t$, we obtain a surjection onto a free subsheaf of $N_{h_t/X_t}$ of the same rank, which shows that the degree of $N_{h_t/X_t}$ is at least the degree of $\mathcal{K}|_{Y_t}$.

By Lemma \ref{lemma}, we have the last inequality.
\end{proof}
\end{lemma}

\section{Weighted projective space} 
\subsection{Preliminaries} 

We recall some basic facts about weighted projective spaces; standard references include \cite{Dolgachev1982,Fletcher1989,Fulton1993}.
\begin{definition}
    A weighted projective space $\mathbb P(a_0, ... , a_n)$ is \emph{well-formed} if for each $i$, we have that $\mathrm{gcd}(a_0,...,\widehat{a_i},...,a_n)=1$.
\end{definition}

Any weighted projective space is isomorphic to a well-formed weighted projective space. Given a well-formed weighted projective space, the singularities are distinguished by toric strata, and are determined by the following: for any $I\subset \{0,...,n\}$, if $g:=\mathrm{gcd}(a_i, i\in I)\neq 1$, then each point $p$ in the toric stratum $x_i\neq 0, i\in I$ and $x_i =0, i\notin I$ in
a neighborhood is analytically  $\mathbb C^{|I|-1}\times (\mathbb C^{\,n-|I|+1}/\mu_g)$, where for $\mathbb C^{\,n-|I|+1}/\mu_g$ the action is $\zeta\mapsto(\zeta^{a_0} x_0, \zeta^{a_1} x_1,...\widehat{\zeta^{a_i} x_i},..., \zeta^{a_n} x_n)$ omitting all $i\in I$ and we denote it by $\frac 1g(a_0 , ..., \widehat{a_i},... ,a_n)$. For example, in $\mathbb{P}(1,1, 2,4)$, the singularities are given by the toric strata $(0,0,1,0)$, $(0,0,0,1)$, $(0,0,*,*)$ where $*$ represents nonzero numbers.

We consider specifically in this paper about weighted projective spaces $\mathbb{P}(a_0,\ldots,a_n)$ with isolated singularities, so the singular points are among the torus fixed points $(1,0,...,0), (0,1,0...,0)$ $,...,(0,...,0,1)$. We see that if it is well-formed, then this is equivalent to say that the parameters are pairwise coprime.  

The Picard group of a general well-formed $\mathbb{P}(a_0,\ldots,a_n)$ is:
\[ \operatorname{Pic}(\mathbb{P}(a_0,\ldots,a_n)) = \mathbb{Z} \cdot \mathcal{O}(k) \text{ where } k = \operatorname{lcm}(a_0,\ldots,a_n). \]
For weighted projective spaces with isolated singularities, we have $k = \prod_i a_i$.
\begin{lemma}[{\cite[Proposition 2.1, Example 5.1]{Keum1997}}]
Let $\mathbb P(a_0,...,a_n)$ be a weighted projective space, then $\mathcal O(k)$ is a very ample line bundle and $\mathcal O(k)$ is normally generated. 
\end{lemma}

\begin{theorem}[{\cite[p.~48]{Fulton1993}}]
    For any toric variety $X(\Delta)$, there is a refinement $\widetilde{\Delta}$ of $\Delta$ so that $X(\widetilde{\Delta})\to X(\Delta)$ is a resolution of singularities.
\end{theorem}
\begin{remark}
In particular, for a weighted projective space, there exists a toric resolution $f:\widetilde{\mathbb P}\to \mathbb P$. This shows that after resolving the singularities, we still get a toric variety, which satisfies the condition to apply the technique in the previous section.
\end{remark}

Sometimes, it may be helpful to understand the Picard group of the resolution $\widetilde{\mathbb P}$ so that one may gain a better understanding on the Picard group of its hypersurfaces, which potentially provides further information on algebraic hyperbolicity; see the proof of \cite[Proposition 3.10]{coskun2019algebraichyperbolicitygeneralsurfaces}. But in general, calculating the Picard group is delicate as it depends on the cyclic quotient singularities on the weighted projective spaces. 

\begin{example}  At the end of \cite{coskun2019algebraichyperbolicitygeneralsurfaces}, Coskun and Riedl consider $\mathbb{P}(1,1,1,n)$. The local depiction around $(0,0,0,1)$ is $\mathbb{C}^3/\mathbb{Z}_n$ acting by $(\zeta_n x, \zeta_n y, \zeta_n z)$ or say $\frac{1}{n} (1,1,1)$.   This singularity is special because it is actually a cone singularity: We have that the quotient is given by $\operatorname{Spec } \mathbb{C}[x,y,z]^{\mathbb Z_n} = \operatorname{Spec } \mathbb{C}[x^n, x^{n-1}y, x^{n-1}z,\ldots,z^n]$. This is exactly the affine cone of $\mathbb{P}^2$ under the Veronese embedding. We know all cone singularities can be resolved by a single blow-up of the vertex. So one may get that the Picard group of $\widetilde{\mathbb P}$ is generated by $H:= f^*\mathcal O(n)$ and $F$ where $nF:= H-E$, $E$ is the only exceptional divisor. But this case is very special. In general, it requires multiple blow-ups to resolve an isolated singularity, and non-isolated singularities only make the situation more complex.
\end{example}

\subsection{Weighted projective space with isolated singularities}
We consider a well-formed weighted projective space $\mathbb{P}(a_0,a_1,\dots,a_n)$ of dimension $n\geq 3$. We have that the weights are pairwise coprime and $\operatorname{Pic}(\mathbb{P}) = \mathbb{Z} \mathcal{O}(k )$, where $k=a_0a_1\dots a_n$. Consider a toric resolution $f: \widetilde{\mathbb{P}} \to \mathbb{P}$. Let $H = f^*\mathcal{O}(k)$. We would like to find a lower bound of $m$ such that the zero locus of a very general section of $mH, m\geq 1$ is algebraically hyperbolic.

\begin{lemma}
$H^0(\mathbb{P}, \mathcal O(mk)) \cong  H^0(\widetilde{\mathbb{P}}, mH)$ and a general hypersurface of \(H^0(\mathbb{P}, \mathcal O(mk))\) is isomorphic to a general hypersurface of \(H^0(\widetilde{\mathbb{P}}, mH)\).
\end{lemma}
\begin{proof}
    
By the projection formula we have that $f_*(f^*\mathcal{O}(k)\otimes \mathcal{O}_{\widetilde{\mathbb{P}}}) = \mathcal{O}(k)\otimes f_*\mathcal{O}_{\widetilde{\mathbb{P}}} = \mathcal{O}(k)\otimes \mathcal{O}_{\mathbb{P}}$, so $f_*f^*\mathcal{O}(k) = \mathcal{O}(k)$, so we have that $H^0(\widetilde{\mathbb{P}}, f^*\mathcal{O}(k)) \cong H^0(\mathbb{P}, \mathcal{O}(k))$. Also $\mathcal O(k)$ is base-point-free, so for any point $p$ there exists one section and thus general sections avoid $p$. Since the weighted projective space may only have singularities among points $(1,\dots,0)$, $(0,1,0,\dots,0)$,$\dots$,$(0,\dots,0,1)$, we have that the zero locus of a general section $s$ in $H^0(\mathbb{P}, \mathcal{O}(k))$ avoids all singular points, and so $Z(s)$ is isomorphic to $Z(f^*s)$.
\end{proof}
We can also find the Picard group of the zero locus of a general section of $H^0(\widetilde{\mathbb{P}}, mH)$.
Even though this information is not required.

\begin{theorem}[{\cite[Theorem 1]{ravindra2005grothendiecklefschetztheoremnormalprojective}}]
Let $X$ be an irreducible projective variety over an algebraically closed field of characteristic 0, regular in codimension 1, and let $L$ be an ample line bundle over $X$, together with a linear subspace $V \subset H^0(X, L)$ which gives a base-point-free ample linear system $|V|$ on $X$. For a dense Zariski open set of $Y \in |V|$, the restriction map $\operatorname{Cl}(X) \to \operatorname{Cl}(Y)$ is an isomorphism if $\dim X \geq 4$, and is injective with finitely generated cokernel if $\dim X = 3$.
\end{theorem}
\begin{lemma}
Let $X$ be the zero locus of a general section of $mH$ in $\widetilde{\mathbb{P}}$, then $\operatorname{Pic}(X)\otimes \mathbb{Q} = \mathbb{Q} H$ if $\operatorname{dim } \mathbb P\geq4$.
\end{lemma}
\begin{proof}
    Since $\mathcal O(mk)$ on $\mathbb P$ is very ample, we apply the theorem above, and we have that for a general element $X$, $\operatorname{Cl}(\mathbb{P}) = \operatorname{Cl}(X)$. Since a general $X$ is smooth by Bertini's Theorem applied to $\widetilde{\mathbb P}$, we have $\operatorname{Cl}(X) = \operatorname{Pic}(X)$. Also we have $\mathbb{P}$ is $\mathbb{Q}$-factorial so $\operatorname{Pic}(X) \otimes \mathbb{Q} = \operatorname{Cl}(\mathbb P)\otimes \mathbb Q = \operatorname{Pic}(\mathbb P) \otimes \mathbb{Q} = \mathbb{Q}H$.
\end{proof}

\begin{lemma}
 $mH$ is section-dominated by $H$, for $m\geq 1$.
\end{lemma}
\begin{proof}
The section ring of $H$ is isomorphic to the section ring of $\mathcal O(k)$, which is normally generated. So, by Lemma \ref{section}, $mH$ is section-dominated by $H$.
\end{proof}
\begin{definition}
 A complex projective variety is \emph{algebraically hyperbolic outside a sub-variety $Z \subsetneq X$} if there exists $\epsilon > 0$ and an ample line bundle $P$ such that $2g(C)-2 \geq\epsilon\cdot\deg_P C$ for any integral curve $C
 \subset X$ not contained in $Z$, where $g(C)$ is the geometric genus of $C$. 
\end{definition}

\begin{proposition}\label{prop2}
 If $m > \frac{a_0 + \dots + a_n}{a_0a_1... a_n} + (n-2)$, then a very general hypersurface $X$ of $|mH|$ is algebraically hyperbolic outside the toric boundary.
\end{proposition}

\begin{proof}
For a very general hypersurface $X$ of $|mH|$, we can identify it as a very general hypersurface in $|\mathcal O(mk)|$. We have:
\[ K_X = (K_{\mathbb{P}} + X)|_X = (K_{\mathbb{P}} + mH)|_X \]

\[ K_X = \left(m - \frac{a_0 + \dots + a_n}{a_0a_1... a_n}\right)H|_X \]
where $K_{\mathbb{P}} = \mathcal O(-a_0-a_1-...-a_n)$.

By Lemma \ref{lem}, for a curve $C\subset X$ not entirely contained in the toric boundary of $\widetilde{\mathbb P}$, we have that
\[ 2g - 2 \geq -(n-2)\deg H|_C + K_X\cdot C=-(n-2)H \cdot C + \left(m - \frac{a_0 + \dots + a_n}{a_0a_1... a_n}\right)H \cdot C \]
\[= \left(-(n-2) + m - \frac{a_0 + \dots + a_n}{a_0a_1... a_n}\right) \deg_H C\]

If $m > \frac{a_0 + \dots + a_n}{a_0a_1... a_n} + (n-2)$, we have $X$ is algebraically hyperbolic outside the toric boundary.
\end{proof}

\begin{remark}
If we don't use the fact that $mH$ being section dominated by $H$, then from Lemma \ref{lem2}, we get that $2g-2 \geq \left(-(n-2)m +m-\frac{a_0 + \dots + a_n}{a_0a_1... a_n} \right)\deg_HC$, making  $-(n-2)m +m-\frac{a_0 + \dots + a_n}{a_0a_1... a_n} $ negative for all $m\geq 1$.
\end{remark}

\begin{example}For $\mathbb P(1,1,2,3,5)
$, $m> (1+1+2+3+5)/(1\cdot1\cdot2\cdot 3\cdot 5) +2 = 12/5$, so $m\geq 3$ would make a very general $X$ algebraically hyperbolic outside its intersection with the toric boundary,  which is $\mathbb P(1,2,3,5)\cup \mathbb P(1,2,3,5)\cup \mathbb P(1,1,3,5)\cup \mathbb P(1,1,2,5)\cup \mathbb P(1,1,2,3)$.
\end{example}

We recall the bounds of $m$ for $\mathbb P(1,1,1,n)$ in Coskun and Riedl \cite{coskun2019algebraichyperbolicitygeneralsurfaces}.

\begin{proposition}
[{\cite[Lemma 3.9, Proposition 3.10]{coskun2019algebraichyperbolicitygeneralsurfaces}}]\label{prop1}
A very general surface $X \in |mH|$ is algebraically hyperbolic if $m \geq 4$ and $n \geq 2$; or $m = 3$ and $n \geq 4$; or $m = 2$ and $n \geq 5$. And $X$ will not be algebraically hyperbolic if $n = 1$, $m\leq4$ or $m = 2$, $n \leq 4$.
\end{proposition}

We generalize the result to arbitrary weighted projective spaces with isolated singularities. 

\begin{proposition}\label{prop3}
For a weighted projective $3$-fold $\mathbb{P}(a_0,a_1,a_2,a_3)$ with isolated singularities, if $\mathbb{P}(a_0,a_1,a_2,a_3) \neq \mathbb{P}(1,1,1,n)$ or $\mathbb{P}(1,1,2,3)$, then a very general surface $X\in |mH|$ is algebraically hyperbolic if $m \geq 2$. For $\mathbb{P}(1,1,2,3)$, a very general surface $X\in|mH|$ is algebraically hyperbolic if $m\geq 3$.
\end{proposition}

\begin{proof}
    Let $X$ be a very general surface of $|mH|$ of a weighted projective 3-fold with isolated singularities $\mathbb P(a_0,a_1,a_2,a_3)$. Then by Proposition \ref{prop2}, we have that if $m > (a_0 + a_1+a_2 + a_3)/a_0a_1a_2 a_3 + 1$, then $X$ is algebraically hyperbolic outside the toric boundary.

    We note that $a_0+a_1+a_2+a_3 < a_0a_1a_2a_3$ if $\mathbb{P}(a_0, a_1, a_2, a_3) \neq \mathbb{P}(1, 1, 1, n)$ or $\mathbb{P}(1, 1, 2, 3)$. To show this, assume here $a_0 \leq a_1 \leq a_2 \leq a_3$. If $a_2 \geq 4$, then $a_2a_3 > a_0+a_1+a_2+a_3$. If $a_2 < 4$, then we only have the case $a_2 = 2,3$, leading to $\mathbb{P}(1,1,2,a_3)$ with $a_3>3$ (so $a_3 \geq 5$), $\mathbb{P}(1,1,3,a_3)$, and $\mathbb{P}(1,2,3,a_3)$. In each of these cases, the inequality holds. So for $\mathbb{P}(a_0, a_1, a_2, a_3) \neq \mathbb{P}(1, 1, 1, n)$ or $\mathbb{P}(1, 1, 2, 3)$, we have $m \geq 2$. For $\mathbb{P}(1,1,2,3)$, we have $m \geq 3$.

    It remains to show that the intersection of $X$ with the toric boundary is algebraically hyperbolic. Note that the intersection is a union of finitely many curves, so it suffices to show there are no rational and elliptic curves in the intersection. We note that a very general $X$ intersects any toric invariant divisor at a smooth curve, so it suffices to show that the canonical divisor of that curve has positive degree. Each toric invariant divisor is a weighted projective surface with one coefficient removed, $\mathbb{P}(a_0,\dots,\widehat{a_i},\dots,a_3)$. Say $a_0$ is removed, then, $K_C = (K_{\mathbb{P}(a_1,a_2,a_3)} + \mathcal O(mk))|_C =\mathcal O(-a_1-a_2-a_3 + m a_0a_1a_2a_3)|_C$.  We have $a_1+a_2+a_3 < m a_0a_1a_2a_3$ for $m \geq 2$.
\end{proof}
\begin{remark}
We know $m=1$ for $\mathbb P(1,1,2,3)$, $X$ is not algebraically hyperbolic. So the remaining cases for 3-dimensional weighted projective spaces are $m = 3$ for $\mathbb P(1,1,1,2)$ and $\mathbb P(1,1,1,3)$, $m=2$ for $\mathbb P(1,1,2,3)$, and $m=1$ for $\mathbb P(a_0,a_1,a_2,a_3)\neq \mathbb P(1,1,1,n) \text{ or } \mathbb P(1,1,2,3)$.
\end{remark}

\begin{proposition}\label{prop4}
  
Let $\mathbb{P} = \mathbb{P}(a_0,\dots,a_n)$ be a weighted projective space with isolated singularities. Let
\[
\Theta := \max_{\substack{I \subset \{0,\dots,n\} \\ |I|\ge 4}} \left\{\frac1{\prod_{i\notin I}a_i} \left( \frac{\sum_{i \in I} a_i}{\prod_{i\in I} a_i} + (|I|-3) \right)\right\}.
\]
Then for every $m > \Theta$, a very general hypersurface $X \in |mH|$ is algebraically hyperbolic.
\end{proposition}

\begin{proof}
Since we work with $\mathbb P$ and its toric strata, we consider a very general $X$ of $mH$ in $\widetilde{\mathbb P}$ to be a very general $X$ of $\mathcal O(mk)$ in $\mathbb P$. We note that for a weighted projective space $\mathbb{P}(a_0,\dots,a_n)$, the toric boundary divisors are exactly 
\[
D_i = \mathbb{P}(a_0,\dots,\widehat{a_i},\dots,a_n), \quad i \in \{0,\dots,n\}.
\] 
Let $I \subset \{0,\dots,n\}$, $|I|\geq 4$, then we have $\mathbb{P}(a_i, i \in I)$ as a stratum of $\mathbb{P}(a_0,\dots,a_n)$ in the sense of hyperbolicity. That is, if we show that for each $I \subset \{0,\dots,n\}$ with $|I|\geq 5$, a very general section of $\mathcal O(mk)|_{\mathbb{P}_I}$ in $\mathbb{P}_I$ is algebraically hyperbolic with respect to the ample $\mathcal O(k)|_{\mathbb{P}_I}$ outside the toric boundary of $\mathbb{P}_I$, and for $|I|=4$ a very general section of $\mathcal O(mk)|_{\mathbb{P}_I}$ in $\mathbb{P}_I$ is algebraically hyperbolic with respect to $\mathcal O(k)|_{\mathbb P_I}$. Then since in our case, the restriction maps of the global sections of $\mathcal O(mk)$ to each stratum $H^{0}\!\big(\mathbb{P},\, \mathcal{O}_{\mathbb{P}}(mk)\big)
\rightarrow
H^{0}\!\big(\mathbb{P}_{I},\, \mathcal{O}_{\mathbb{P}_{I}}(mk)\big)$ are surjective. So taking the finite intersection of preimages of very general sets one in each $H^{0}\!\big(\mathbb{P}_{I},\, \mathcal{O}_{\mathbb{P}_{I}}(mk)\big)$, we get a very general set in $H^{0}\!\big(\mathbb{P},\, \mathcal{O}_{\mathbb{P}}(mk)\big)$, such that for any hypersurface $X$ in the set, $\mathbb P_I \cap X$ is algebraically hyperbolic for each $|I| = 4$ and $\mathbb P_I \cap X$ is algebraically hyperbolic outside the toric boundary for each $|I| \geq 5$, so together they imply that $X$ is algebraically hyperbolic. We note that each stratum $\mathbb P_I$ is a weighted projective space with isolated singularities, so we may apply Proposition \ref{prop2}, \ref{prop3}. To find an $\epsilon$, we take the minimum $\epsilon_I$ of each stratum. We note that algebraic hyperbolicity is independent of ample divisor, but here we fix the ample divisor $\mathcal O(k)$ to give a global $\epsilon$ for a very general $X$. 

We see that the inequality $m>\Theta$ produces a sufficient condition for algebraic hyperbolicity. This is because by Proposition \ref{prop3},
\[
m > \frac{1}{\prod_{i\notin I}a_i}(\frac{\sum_{i \in I} a_i}{\prod_{i\in I} a_i} + (|I|-3)), \quad \text{for } |I|=4,
\] 
means that a very general section of $mq_I\mathcal O(k_I)  = \mathcal O(mk)|_{\mathbb{P}_I}$, where $\mathcal O(k_I):= \mathcal O(\prod_{i\in I} a_i)$ and $q_I := \prod_{i \notin I} a_i$, is algebraically hyperbolic with respect to $\mathcal O(k_I)$. So this implies a very general section of $\mathcal O(mk)|_{\mathbb{P}_I}$ is algebraically hyperbolic with respect to $\mathcal O(k)|_{\mathbb{P}_I}$ with $\epsilon_I = \frac 1{q_I}(\frac{\sum_{i \in I} a_i}{\prod_{i\in I} a_i} + (|I|-3))$. Similarly, for $|I|\geq5$ by Proposition \ref{prop2},
\[
m >  \frac{1}{\prod_{i\notin I}a_i}(\frac{\sum_{i \in I} a_i}{\prod_{i\in I} a_i} + (|I|-3)),
\] 
implies that a very general section of $mq_I\mathcal O(k_I)  = \mathcal O(mk)|_{\mathbb{P}_I}$ is algebraically outside the toric boundary of $\mathbb P_I$ with respect to $\mathcal O(k)|_{\mathbb{P}_I}$ with $\epsilon_I = \frac 1{q_I}(\frac{\sum_{i \in I} a_i}{\prod_{i\in I} a_i} + (|I|-3))$. So we get a very general section of $\mathcal O(mk)$ is algebraically hyperbolic with respect to $\mathcal O(k)$ with an $\epsilon = \operatorname{min}\{\epsilon_I, |I|\geq4\}$.
   
\end{proof}

\begin{corollary}\label{cor:uniform}
For $n\ge3$, if $m\ge 2n$, then 
$X$ is algebraically hyperbolic. If at least one weight $a_i\ge2$, then
for $m> \frac{3n}{2}-1$, $X$ is algebraically hyperbolic. Let $\mathbb{P}=\mathbb{P}(1,\ldots,1,n)$ be a weighted projective space of dimension $l\ge3$. If $
m>\,l-1+\frac{l}{n},$
then $X$ is algebraically hyperbolic.
\end{corollary}

\begin{proof}
Since $\prod_{i\notin I}a_i\ge1$, we have
\[
\Theta \le \max_{|I|\ge4}\Bigl(\frac{\sum_{i\in I}a_i}{\prod_{i\in I}a_i}+(|I|-3)\Bigr)
\le \max_{|I|\ge4}\bigl(|I|+(|I|-3)\bigr)=2n-1.
\]
If not all $a_i=1$, the worst case for $\frac{\sum_{i\in I}a_i}{\prod_{i\in I}a_i}$ is
one weight is $2$ and the other are $1$, giving
$\frac{\sum a_i}{\prod a_i}\le\frac{|I|+1}{2}$. Hence
\[
\Theta \le \max_{|I|\ge4}\bigl(\frac{|I|+1}{2}+(|I|-3)\bigr) = \max_{|I|\ge4}\bigl(\frac{3|I|-5}{2}\bigr)
= \frac{3(n+1)-5}{2}=\frac{3n}{2}-1.
\]

For $\mathbb P(1,1,...,1,n)$ of dimension $l$, we have that $\Theta = l-1+\frac l n$, so for $m> l-1+\frac ln$, we have $X$ is algebraically hyperbolic.
\end{proof}

\begin{remark}
Like in the 3–dimensional case, we may ask for which values of \(m\) a very general
hypersurface \(X \subset \PP(a_0,\ldots,a_n)\) is not algebraically hyperbolic.
A possible approach is to check for lower dimensional strata.

\end{remark}
\begin{remark}
    One may generalize the analysis to weighted projective spaces with non-isolated singularities, we can apply the technique to a toric resolution of $\mathbb P$, $f:\widetilde{\mathbb P}\to\mathbb P$ and consider the line bundle $mH=mf^*\mathcal O(k)$, where $k=\operatorname{lcm}(a_0,...a_n)$, but
    \begin{itemize}
        \item A general hypersurface of $|mH|$ will not be isomorphic to $|\mathcal O(mk)|$ as a general hypersurface of $|\mathcal O(mk)|$ will intersect with the singularity locus non-trivially. So it may not be meaningful as one would like to examine hyperbolicity of a very general hypersurface of $|\mathcal O(mk)|$. 
        \item The canonical divisor $K_X$ would be complicated and cannot be written in terms of $H$ only. If we consider a toric crepant resolution, then we get $K_X = (f^*K_\mathbb P+mH)|_X = (m-\frac{a_0+...+a_n}{\operatorname{lcm}(a_0,...,a_n)})H|_X$. But not all weighted projective spaces admit a toric crepant resolution.
        \item The pullback $H|_X$ in general is base-point-free and big but may not be ample, so if one would like to find a polarization for $X$, one has to find some ample divisor other than $H$.
    \end{itemize}
\end{remark}

\bibliographystyle{plain}
\bibliography{reference}

\end{document}